\documentclass[11pt,reqno]{amsart}
\usepackage[foot]{amsaddr}
\usepackage{amsfonts,amssymb,mathrsfs,amscd,comment,amsthm}
\usepackage{mathptmx}      

\newcommand{\con}{\operatorname{con}}

\newcommand{\occ}{\operatorname{occ}}

\newtheorem*{theorem}{Theorem}
\newtheorem{lemma}{Lemma}

\theoremstyle{remark}
\newtheorem{remark}{Remark}

\usepackage{cite}

\DeclareSymbolFont{rsfscript}{OMS}{rsfs}{m}{n}
\DeclareSymbolFontAlphabet{\mathrsfs}{rsfscript}

\makeatletter
\@namedef{subjclassname@2020}{\textup{2020} Mathematics Subject Classification}
\makeatother

\begin{document}

\title{Identity bases for finite cyclic semigroups}

\author[M. V. Volkov]{Mikhail V. Volkov}
\address{{\normalfont 620075 Ekaterinburg, Russia}}
\email{m.v.volkov@urfu.ru}

\date{}

\thanks{Supported by the Ministry of Science and Higher Education of the Russian Federation, project FEUZ-2023-2022}


\begin{abstract}
We provide explicit identity bases for finite cyclic semigroups.
\end{abstract}

\maketitle

\subsection*{Background} The result reported in the present note was found in 1984, when I was collecting material for the forthcoming survey paper on semigroup identities~\cite{ShVo85}. One chapter of that survey was planned as a catalogue of identity bases for concrete ``interesting'' semigroups. To my surprise, I could not find any information in the literature on what I thought to be the simplest case---finite cyclic semigroups (apart from finite cyclic groups and cyclic semigroups with at most four elements). This looked like a gap worth filling, and I tried to do so, but the task turned out to be far from an easy exercise. A solution came after I spotted a key ingredient in a then-recent paper by Lyapin~\cite{Lyapin:1979}; see Lemma~\ref{lem:lyapin} below. I had the opportunity to mention this to Lyapin during the conference on the occasion of his 70th birthday in September 1984.

The result subsequently appeared in~\cite{ShVo85} as Proposition~21.3---and has since been used in the literature; see \cite[Proposition 5.8]{LRS:2019},  for instance---but its proof has not been published until now. The present collection of papers dedicated to Lyapin's memory seems to be an appropriate venue for this note, built upon his idea.

\subsection*{Cyclic semigroups} A semigroup is called \emph{cyclic} (or \emph{monogenic}) if it can be generated by a single element. If a cyclic semigroup is infinite, then it is isomorphic to the semigroup of positive integers under addition. Every finite cyclic semigroup is uniquely determined by two positive integers---index $h$ and period $d$---via the following presentation:
\[
C_{h,d}: = \langle a\mid a^{\,h} = a^{\,h+d}\rangle.
\]
See \cite[Section 1.6]{ClPr61} for details and historical information.

The Cayley graph of the semigroup  $C_{h,d}$ with respect to its generator $a$ is shown in Fig.~\ref{fig:Cayley}. The graph provides a compact depiction of the structure of $C_{h,d}$. If one thinks of the graph as a frying pan, the $d$-element base $\{a^{\,h+1},a^{\,h+2},\dots,a^{\,h+d}\}$ forms an ideal in $C_{h,d}$, which is isomorphic to the cyclic group $C_{1,d}$. The $h$-element handle $\{a,a^2,\dots,a^{\,h}\}$ corresponds to the Rees quotient of the semigroup $C_{h,d}$ over this ideal; this quotient is isomorphic to  the nilpotent cyclic semigroup $C_{h,1}$. It is well known (and easy to verify) that the semigroup $C_{h,d}$ is a subdirect product of $C_{h,1}$ and $C_{1,d}$.
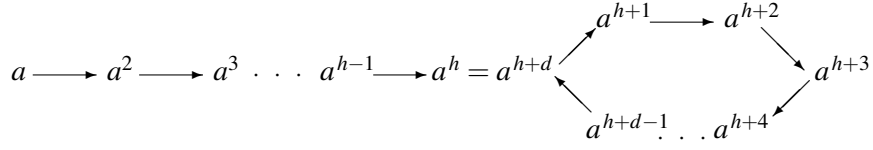
\begin{figure}[ht]
\begin{center}
\unitlength=1.4mm
\begin{picture}(90,13)(0,2)
\put(6,8.5){$a$}
\put(15,8.5){$a^2$}
\put(25,8.5){$a^3$}
\put(8,9){\vector(1,0){6}}
\put(18,9){\vector(1,0){6}}
\multiput(29,9)(2,0){3}{\circle*{0.2}}
\put(35,8.5){$a^{\,h-1}$}
\put(45.5,8.5){$a^{\,h}=a^{\,h+d}$}
\put(40,9){\vector(1,0){5}}
\put(57.5,10){\vector(1,1){3.4}}
\put(66,14){\vector(1,0){6}}
\put(76.5,13.5){\vector(1,-1){4}}
\put(81,8.5){\vector(-1,-1){3.4}}
\multiput(67,3)(2,0){3}{\circle*{0.2}}
\put(60.5,5.5){\vector(-1,1){3.4}}
\put(60,3){$a^{\,h+d-1}$}
\put(61,13.5){$a^{\,h+1}$}
\put(73,13.5){$a^{\,h+2}$}
\put(81.5,8.5){$a^{\,h+3}$}
\put(72,3){$a^{\,h+4}$}
\end{picture}
\end{center}
\caption{The Cayley graph of the semigroup $C_{h,d}$}\label{fig:Cayley}
\end{figure}

\subsection*{Words and identities} We fix a countably infinite set $X:=\{x,y,x_1,x_2,\dots\}$; its elements are called \emph{letters}. A \emph{word} is a finite sequence of letters. If $w$ is a word, the set of distinct letters occurring in $w$ is called the \emph{content} of $w$ and is denoted $\con(w)$. For a letter $x$ and a word $w$, we denote by $\occ(x,w)$ the number of its occurrences in $w$, that is, if $w=x_1\cdots x_\ell$, where $x_1,\dots,x_\ell$ are letters, then $\occ(x,w)$ is the cardinality of the set $\{i\mid x_i=x\}$. The number $|w|:=\sum_{x\in\con(w)}\occ(x,w)$ is called the \emph{length} of $w$. We allow the \emph{empty word}, which has empty content and length 0. In particular, a power of a letter with exponent 0 is understood as the empty word.

Any map $\varphi\colon X\to S$, where $S$ is semigroup is called a \emph{substitution}. The \emph{value} $w\varphi$ of a non-empty word $w=x_1\cdots x_\ell$ under $\varphi$ is defined as $x_1\varphi\cdots x_\ell\varphi$.

An \emph{identity} is an expression of the form $u=v$, where $u$ and $v$ are non-empty words. A semigroup $S$ \emph{satisfies} $u=v$ (or $u=v$ \emph{holds} in $S$) if $u\varphi=v\varphi$ for every substitution $\varphi\colon X\to S$. That is, each substitution of elements from $S$ for the letters occurring in $u$ or $v$ yields equal values for these words.

Given a system $\Sigma$ of identities, one can deduce its \emph{consequences} by repeatedly applying the following rules: substituting a word for each occurrence of a variable in an identity, multiplying an identity on the left or right by a word, and using the symmetry and transitivity of equality. Two identities are said to be \emph{equivalent modulo} $\Sigma$ if each of them is a consequence of the other together with $\Sigma$.

A system $\Sigma$ of identities holding in a semigroup $S$ is called an \emph{identity basis} for~$S$ if every identity satisfied by $S$ is a consequence of $\Sigma$.

A letter $x$ is said to be \emph{balanced} in an identity $u=v$ if $\occ(x,u)=\occ(x,v)$; otherwise, $x$ is \emph{unbalanced} in $u=v$. An identity $u=v$ is \emph{balanced} if every letter is balanced in it and \emph{non-balanced} if it has an unbalanced letter. It is well known (and easy to verify) that an identity is balanced if and only if it is a consequence of the commutative law
\begin{equation}
\label{eq:com}
xy = yx.
\end{equation}
The law \eqref{eq:com} forms an identity basis for the infinite cyclic semigroup.

\subsection*{Identities of finite cyclic semigroups}
As every cyclic semigroup is commutative, it satisfies all balanced identities. Therefore, we restrict our attention to the non-balanced identities of finite cyclic semigroups. The following result, due to Lyapin, classifies such identities in a crucial special case.

\begin{lemma}[{\!\cite[Lemma 17]{Lyapin:1979}}]\label{lem:lyapin}
A non-balanced identity $u=v$ holds in the nilpotent cyclic semigroup $C_{h,1}$ if and only if either \emph{(a)} $|u|,|v|\ge h$, or \emph{(b)} $\con(u)=\con(v)$ and $|u|=|v|\ge h-\min_x\{\occ(x,u),\occ(x,v)\mid \occ(x,u)\ne\occ(x,v)\}$.
\end{lemma}

For future use, we call the non-balanced identities of type (a) \emph{long} and those of type (b) \emph{uniform}.

We call $u=v$ a $d$-\emph{balanced} identity if $\occ(x,u)\equiv\occ(x,v)\pmod{d}$ for every letter $x$. The next result is a part of semigroup folklore,

\begin{lemma}\label{lem:group}
An identity holds in the cyclic group $C_{1,d}$ if and only if it is $d$-balanced.
\end{lemma}

Combining Lemmas~\ref{lem:lyapin} and~\ref{lem:group}, we classify non-balanced identities of arbitrary finite cyclic semigroups.

\begin{lemma}\label{lem:combine}
\emph{(i)} A non-balanced identity $u=v$  holds in the semigroup $C_{h,d}$ with $h>d+2$ if and only if it is $d$-balanced and either long or uniform.

\emph{(ii)} A non-balanced identity $u=v$ holds in the semigroup $C_{h,d}$ with $h\le d+2$ if and only if it is $d$-balanced and long.
\end{lemma}

\begin{proof}
As mentioned, the semigroup $C_{h,d}$ is a subdirect product of $C_{h,1}$ and $C_{1,d}$. An identity holds in a subdirect product if and only if it holds in each factor. Therefore, a non-balanced identity $u=v$ holds in $C_{h,d}$ if and only if it simultaneously satisfies the conditions in Lemmas~\ref{lem:lyapin} and~\ref{lem:group}, that is, the identity is $d$-balanced and either long or uniform. This yields claim (i) and the ``if'' part of claim (ii). For the ``only if'' part of claim (ii), we prove that if $h\le d+2$, then every $d$-balanced uniform identity $u=v$ is long so the second alternative  in ``either long or uniform'' is subsumed by the first.

The uniformity of $u=v$ implies $\con(u)=\con(v)$ and $|u|=|v|$. Let $x$ be an unbalanced letter in $u=v$. Then the non-negative integers $\occ(x,u)$ and $\occ(x,v)$ are distinct whence at least one of them is not 0. The equality $\con(u)=\con(v)$ rules out the case where one of these integers is 0 and the other is not. Hence, $\occ(x,u),\occ(x,v)\ge 1$. Because $u=v$ is $d$-balanced, we have $\occ(x,u)\equiv\occ(x,v)\pmod{d}$, whence
\[
\max\{\occ(x,u),\occ(x,v)\}\ge d+1.
\]
The equality $|u|=|v|$ rules out the case $\con(u)=\con(v)=\{x\}$ (otherwise $|u|=\occ(x,u)\ne\occ(x,v)=|v|$). Thus, each of the words $u$ and $v$ contains at least one occurrence of a letter different from $x$. Consequently, the word with more $x$'s has length at least $(d+1)+1=d+2$. Since $|u|=|v|$, we obtain $|u|,|v|\ge d+2\ge h$, showing that the identity $u=v$ is long.
\end{proof}

\subsection*{Identity bases of finite cyclic semigroups}
\begin{theorem}
An identity basis for the semigroup $C_{h,d}$ is formed by the identities \eqref{eq:com} and
\[
\Phi_{h,d}:\quad x^{\,d}x_1\cdots x_h = x_1\cdots x_h,
\]
if $h\le d+2$; otherwise, if $h>d+2$, it is formed by the identities \eqref{eq:com}, $\Phi_{h,d}$, and
\[
\Psi_{h,d,r}:\quad x^{\,r}y^{\,r+d}x_1\cdots x_{h-3r-d} = x^{\,r+d}y^{\,r}x_1\cdots x_{h-3r-d},\ \text{ for each $r=1,\dots, \left\lfloor\dfrac{h-d}3\right\rfloor$}.
\]
\end{theorem}

\begin{remark}
\label{rem:redundancy}
The identity basis for the semigroup $C_{h,d}$ given in \cite[Proposition~21.3]{ShVo85} includes two additional identities:
\begin{gather}
\label{eq:equalpower}
x^{\,d}=y^{\,d}\ \text{ for }\ h\le d,\\
\label{eq:equalpowerwithtail}
y^{\,d}x_1\cdots x_{h-d}=x^{\,d}x_1\cdots x_{h-d}\ \text{ for }\ h>d.
\end{gather}
We now show that both \eqref{eq:equalpower} and \eqref{eq:equalpowerwithtail} can be deduced from \eqref{eq:com} and $\Phi_{h,d}$, and hence, their inclusion in the basis for $C_{h,d}$ was redundant.

To deduce \eqref{eq:equalpower}, substitute the letter $y$ for each of $x_1,\dots,x_h$ in $\Phi_{h,d}$ and then multiply the resulting identity on the right by $y^{\,d-h}$ (which is legitimate since $h\le d$). This yields $x^{\,d}y^{\,d}=y^{\,d}$. Interchanging $x$ and $y$ gives $y^{\,d}x^{\,d}=x^{\,d}$, while the commutative law~\eqref{eq:com} ensures $x^{\,d}y^{\,d}=y^{\,d}x^{\,d}$. Hence $x^{\,d}=y^{\,d}$ by transitivity.

To deduce \eqref{eq:equalpowerwithtail}, make the following substitution in $\Phi_{h,d}$:
\[
x\mapsto x,\quad x_i\mapsto\begin{cases}
y&\text{for }\ i=1,\dots,d,\\
x_{i-d} &\text{for }\ i=d+1,\dots,h.
\end{cases}
\]
The substitution yields the identity $x^{\,d}y^{\,d}x_1\cdots x_{h-d}=y^{\,d}x_1\cdots x_{h-d}$. Interchanging $x$ and $y$ gives $y^{\,d}x^{\,d}x_1\cdots x_{h-d}=x^{\,d}x_1\cdots x_{h-d}$. Since  $x^{\,d}y^{\,d}=y^{\,d}x^{\,d}$ by ~\eqref{eq:com}, we obtain $y^{\,d}x_1\cdots x_{h-d}=x^{\,d}x_1\cdots x_{h-d}$ by transitivity.
\end{remark}

\begin{proof}[Proof of the theorem]
By Lemma~\ref{lem:combine}, the identities from the statement hold in the semigroup $C_{h,d}$ for appropriate $h$ and $d$. Indeed, each of these identities is $d$-balanced. The identity $\Phi_{h,d}$ is long. If $h>d+2$, then the identity $\Psi_{h,d,r}$ is uniform.

Now we show that every identity $u=v$ holding in $C_{h,d}$ follows from the identities~\eqref{eq:com}, $\Phi_{h,d}$, and $\Psi_{h,d,r}$ if $h>d+2$, and from the identities \eqref{eq:com} and $\Phi_{h,d}$ if $h\le d+2$.

We proceed by induction on the number of unbalanced letters in $u=v$. If this number is $0$, then the identity $u=v$ is balanced and follows from~\eqref{eq:com}. Now suppose that $u=v$ is non-balanced. By Lemma~\ref{lem:combine}, the identity is $d$-balanced and either long or uniform, with the latter occurring only when $h>d+2$.

\smallskip

\noindent\emph{\textbf{Case 1:}} the identity $u=v$ is $d$-balanced and long.

First assume that there is a unique unbalanced letter $x$ in $u=v$ so that $\occ(y,u)=\occ(y,v)$ for all $y\ne x$. By symmetry, we may assume that $s:=\occ(x,u)>t:=\occ(x,v)$. Since $s\equiv t\pmod{d}$, we can write $s=t+kd$ for some $k>0$. Using commutativity, we can permute letters in each of the words $u$ and $v$ to obtain an identity of the form $(x^{\,k})^{\,d}x^{\,t}w=x^{\,t}w$, which is equivalent to the identity $u=v$ modulo \eqref{eq:com}. Since permuting letters does not change word length, we have $|x^{\,t}w|=|v|\ge h$. Therefore, the identity $(x^{\,k})^{\,d}x^{\,t}w=x^{\,t}w$ is a consequence of $\Phi_{h,d}$: we obtain the former from the latter by first applying the substitution:
\begin{align*}
x&\mapsto x^{\,k},\\
x_i&\mapsto\text{the $i$-th letter of the word $x^{\,t}w$},\quad i=1,\dots,h,
\end{align*}
and then multiplying both sides of the resulting identity on the right by the suffix of $x^{\,t}w$ that follows the first $h$ letters of this word, that is, by the suffix of length $|x^{\,t}w|-h$. Thus, the identity $u=v$ is a consequence of $\Phi_{h,d}$ and \eqref{eq:com}.

Now let $x,y$ be any two distinct unbalanced letters in $u=v$. By symmetry, we may assume that $p:=\occ(y,u)>q:=\occ(y,v)$. Since $p\equiv q\pmod{d}$ and $p>q\ge 0$, we have $p\ge d$. Using commutativity, we can permute letters in each of the words $u$ and $v$ to rewrite the identity in the form
\[
y^{\,d}y^{\,p-d}u'=y^{\,q}v'.
\]
If $h\le d$, we use the identity \eqref{eq:equalpower} (which follows from~\eqref{eq:com} and $\Phi_{h,d}$; see Remark~\ref{rem:redundancy}) to convert the prefix $y^{\,d}$ of $y^{\,d}y^{\,p-d}u'$ into $x^{\,d}$. Thus, we obtain that the identity
\begin{equation}\label{eq:case1step1}
y^{\,p-d}x^{\,d}u'=y^{\,q}v'
\end{equation}
is equivalent to the identity $u=v$ modulo $\Phi_{h,d}$ and \eqref{eq:com}. If $h>d$, we obtain the same identity \eqref{eq:case1step1} by applying the identity \eqref{eq:equalpowerwithtail} (which also follows from~\eqref{eq:com} and $\Phi_{h,d}$; see Remark~\ref{rem:redundancy}) to the word $y^{\,d}y^{\,p-d}u'$, which we can do since $|y^{\,p-d}u'|=|u|-d\ge h-d$.

If $p-d=q$, the letter $y$ becomes balanced in the identity \eqref{eq:case1step1} and so do all letters that were already balanced in $u=v$. Thus, \eqref{eq:case1step1} has fewer unbalanced letters and, by the induction assumption, it follows from the identities \eqref{eq:com}, $\Phi_{h,d}$, and $\Psi_{h,d,r}$ if $h>d+2$, and from the identities \eqref{eq:com} and $\Phi_{h,d}$ if $h\le d+2$.

If $p-d>q$, we can re-use \eqref{eq:equalpower} (if $h\le d$) or \eqref{eq:equalpowerwithtail} (if $h>d$) to obtain the identity
\[
y^{\,p-2d}x^{\,2d}u'=y^{\,q}v',
\]
still equivalent to $u=v$ modulo $\Phi_{h,d}$ and \eqref{eq:com}. Repeating this transformation, we eventually reach an identity equivalent to $u=v$ modulo $\Phi_{h,d}$ and \eqref{eq:com} but having fewer unbalanced letters, and hence, falling under the induction assumption.

\smallskip

\noindent\emph{\textbf{Case 2:}} $h>d+2$, and the identity $u=v$ is $d$-balanced and uniform.

The uniformity of $u=v$ implies $\con(u)=\con(v)$ and $|u|=|v|$. If $|u|\ge h$ then the identity is long and therefore falls under Case~1, so we assume $|u|<h$. Then $r:=h-|v|>0$.

Take an unbalanced letter $x$ in $u=v$ for which the number
\[
s:=\min\{\occ(x,u),\occ(x,v)\}
\]
is the least possible. Note that $s>0$ because $\con(u)=\con(v)$. By symmetry, we may assume that $s=\occ(x,u)$; then $t:=\occ(x,v)>s$. Since $|u|=|v|$, the deficit caused by $x$ occurring fewer times in~$u$ than in $v$ must be compensated by some letter $y$ that occurs more times in~$u$ than in~$v$. If $p:=\occ(y,u)$ and $q:=\occ(y,v)$, we have $p>q$, and $q\ge s$ by the minimality of $s$. Since the identity $u=v$ is $d$-balanced, we have $s\equiv t\pmod{d}$, whence $t\ge s+d$. We have
\[
|v|\ge|x{\,^t}y^{\,q}|=t+q\ge (s+d)+s=2s+d.
\]
The definition of a uniform identity and the choice of $s$ ensure the inequality $|v|\ge h-s$, whence $s\ge h-|v|=r$. Thus,
\[
h-r=|v|\ge 2s+d\ge 2r+d,
\]
and therefore, $r\le\frac{h-d}3$. This inequality makes $r$ a valid parameter for the identity $\Psi_{h,d,r}$, which we will use from now on.

Using commutativity, we can permute letters in each of the words $u$ and $v$ to obtain an identity of the form
\[
x^{\,s}y^{\,p}u'=x^{\,r+d}y^{\,r}x^{\,t-r-d}y^{\,q-r}v',
\]
equivalent to the identity $u=v$ modulo \eqref{eq:com}. (The inequalities $q\ge s\ge r$ and $t\ge s+d$ established in the preceding paragraph justify the expressions $x^{\,t-r-d}$ and $y^{\,q-r}$.) Permuting letters does not change word length, so $|x^{\,r+d}y^{\,r}x^{\,t-r-d}y^{\,q-r}v'|=|v|=h-r$. Therefore, the identity $\Psi_{h,d,r}$ can be applied to the word $x^{\,r+d}y^{\,r}x^{\,t-r-d}y^{\,q-r}v'$, transforming it to $x^{\,r}y^{\,r+d}x^{\,t-r-d}y^{\,q-r}v'$. Thus,  we obtain that the identity
\begin{equation}\label{eq:case2step1}
x^{\,s}y^{\,p}u'=x^{\,r}y^{\,r+d}x^{\,t-r-d}y^{\,q-r}v'
\end{equation}
is equivalent to the identity $u=v$ modulo \eqref{eq:com} and $\Psi_{h,d,r}$. If $s=t-d$, the letter $x$ becomes balanced in \eqref{eq:case2step1}  and so do all letters that were already balanced in $u=v$. Thus, \eqref{eq:case2step1} has fewer unbalanced letters and falls under the induction assumption.

Suppose that $s\ne t-d$, that is, $t-s-d\ne 0$. Recall that  $s\equiv t\pmod{d}$ and $t\ge s+d$. Therefore, if $t-s-d\ne 0$, then $t-s-d$ is a positive  multiple of $d$, whence $t-s-d\ge d$. In this case, we can rewrite \eqref{eq:case2step1} as
\[
x^{\,s}y^{\,p}u'=x^{\,r+d}y^{\,r}x^{\,t-r-2d}y^{\,q-r+d}v'
\]
and apply the identity $\Psi_{h,d,r}$ to the right-hand side. This leads to the identity
\[
x^{\,s}y^{\,p}u'=x^{\,r}y^{\,r+d}x^{\,t-r-2d}y^{\,q-r+d}v',
\]
still equivalent to the identity $u=v$ modulo \eqref{eq:com} and $\Psi_{h,d,r}$. Repeating this transformation, we eventually reach an identity equivalent to $u=v$ modulo \eqref{eq:com} and $\Psi_{h,d,r}$ but having fewer unbalanced letters, and then we use the induction assumption.
\end{proof}

\subsection*{Acknowledgement} The author thanks the anonymous referee for feedback that led to an essential improvement in the proof of the main theorem (Case 2), and Edmond W.H. Lee for valuable suggestions.


\bigskip

{\small

\begin{thebibliography}{9}
 \bibitem{ClPr61}
 A. H. Clifford, G. B. Preston, The Algebraic Theory of Semigroups, Vol. I, American Mathematical Society, Providence, RI (1961)

\bibitem{LRS:2019}
E. W. H. Lee, J. Rhodes, B. Steinberg, Join irreducible semigroups, Int. J. Algebra Computat. \textbf{29}:7, 1249--1310 (2019)

\bibitem{Lyapin:1979}
 E. S. Lyapin, Identities of sequentially annihilating bundles of semigroups, Izv. Vyssh. Uchebn. Zaved. Mat. no.1:38--45 (1979). [Russian; Engl. translation: Soviet Math. (Iz. VUZ) \textbf{23}:1, 30-–35 (1979)]

\bibitem{ShVo85}
L. N. Shevrin, M. V. Volkov, Identities of semigroups, Izv. Vyssh. Uchebn. Zaved. Mat. no.11:3--47 (1985). [Russian; Engl. translation: Soviet Math. (Iz. VUZ) \textbf{29}:11, 1--64 (1985)]
\end{thebibliography}

}
\end{document}